\documentclass{article}
\usepackage{amsthm,amssymb,amsmath}
\usepackage{url}
\usepackage{fullpage}

\newtheorem{theorem}{Theorem}
\newtheorem{lemma}{Lemma}
\newtheorem{proposition}{Proposition}
\newtheorem{corollary}{Corollary}
\theoremstyle{definition}
\newtheorem{definition}{Definition}
\theoremstyle{remark}
\newtheorem{remark}{Remark}
\newtheorem{example}{Example}

\title{Locally parabolic subgroups in Coxeter groups of arbitrary ranks}
\author{Koji Nuida}
\date{Research Center for Information Security (RCIS), National Institute of Advanced Industrial Science and Technology (AIST), AIST Tsukuba Central 2, 1-1-1 Umezono, Tsukuba, Ibaraki 305-8568, Japan \\ \url{k.nuida@aist.go.jp}}
\begin{document}

\maketitle

\begin{abstract}
Despite the significance of the notion of parabolic closures in Coxeter groups of finite ranks, the parabolic closure is not guaranteed to exist as a parabolic subgroup in a general case.
In this paper, first we give a concrete example to clarify that the parabolic closure of even an irreducible reflection subgroup of countable rank does not necessarily exist as a parabolic subgroup.
Then we propose a generalized notion of \lq\lq locally parabolic closure'' by introducing a notion of \lq\lq locally parabolic subgroups'', which involves parabolic ones as a special case, and prove that the locally parabolic closure always exists as a locally parabolic subgroup.
It is a subgroup of parabolic closure, and we give another example to show that the inclusion may be strict in general.
Our result suggests that locally parabolic closure has more natural properties and provides more information than parabolic closure.
We also give a result on maximal locally finite, locally parabolic subgroups in Coxeter groups, which generalizes a similar well-known fact on maximal finite parabolic subgroups.
\end{abstract}

\section{Introduction}
\label{sec:intro}

For a Coxeter group $W$, a parabolic subgroup is a subgroup that is conjugate to some subgroup generated by a subset of the canonical generating set, and the parabolic closure of a subset $X$ is the intersection of all parabolic subgroups of $W$ that contain $X$.
The notion of parabolic closure has played invaluable roles in the group-theoretic studies of Coxeter groups, which are especially getting active in recent years; for example, the conjugacy problem (e.g., \cite{Kra94}) and the isomorphism problem (e.g., \cite{Nui06}) for Coxeter groups.
The significance of parabolic closure in those preceding works depends on the fact that the parabolic closure itself forms a parabolic subgroup again (hence it is the unique minimal parabolic subgroup containing a given subset), for special cases considered in those works; in particular, for cases of $W$ being finitely generated (see e.g., Franzsen and Howlett \cite[Corollary 7]{FH01} for a proof; while recently Qi \cite{Qi07} re-discovered this fact).
However, it is not guaranteed that this desirable property of parabolic closure holds for a general case.
On the other hand, a well-known fact following Franzsen and Howlett \cite[Lemma 8]{FH01} states that each finite subgroup of a finitely generated Coxeter group $W$ is contained in a maximal finite subgroup of $W$, which is parabolic (see Proposition \ref{prop:Coxeter_maximal_finite_subgroup} below).
This property has also played essential roles in preceding studies of the isomorphism problem for finitely generated Coxeter groups (e.g., \cite{FH01}).
However, similarly to the above case of parabolic closure, this property does not necessarily hold for a general case (indeed, in the inductive limit $S_{\infty} = \bigcup_{n \geq 1} S_n$ of finite symmetric groups, which is a Coxeter group of countably infinite rank, there exists no maximal finite subgroup).
These issues have caused difficulty of generalizing the above-mentioned preceding works on finitely generated Coxeter groups to the case of arbitrary Coxeter groups.

In this paper, first we clarify, by giving a concrete example (Example \ref{ex:counterexample}), that the above-mentioned desirable property of parabolic closure does not hold for arbitrary Coxeter groups, even for restricted cases of parabolic closures of irreducible reflection subgroups of countably infinite ranks (to the author's best knowledge, such a non-trivial counterexample has not appeared in the literature).
Then the aim of this paper is to resolve the lack of the above-mentioned desirable properties in arbitrary Coxeter groups, by generalizing the original notions of parabolic subgroups and parabolic closures.

For the purpose, we define a \lq\lq locally parabolic subgroup'' of a Coxeter group $W$ to be a reflection subgroup $G$ such that each finite subset of the canonical generating set of $G$ (which is determined by the independent results of Deodhar \cite{Deo89} and Dyer \cite{Dye90}) is conjugate to a subset of the canonical generating set of $W$ (Definition \ref{defn:locally_parabolic}).
Then we define the \lq\lq locally parabolic closure'' of a subset $X$ of $W$ to be the intersection of all locally parabolic subgroups that contain $X$.
The notion of locally parabolic subgroups involves parabolic subgroups as a special case.
On the other hand, locally parabolic closures coincide with parabolic closures for certain special cases (see Lemma \ref{lem:two_closures}), but in general these two closures are different (see Example \ref{ex:two_closures_differ}).
One of our main theorems (Theorem \ref{thm:intersection_locally_parabolic}) shows that, in contrast to the case of parabolic closure, the locally parabolic closure is always a locally parabolic subgroup again.
Hence the notion of locally parabolic closures possesses a more desirable property than parabolic closures.
Moreover, another main theorem (Theorem \ref{thm:maximal_locally_finite_subgroup}) shows that each locally finite subgroup of a Coxeter group $W$ is contained in a maximal locally finite subgroup of $W$, which is locally parabolic.
In contrast to the above-mentioned fact on maximal finite subgroups where the type of the underlying Coxeter group $W$ is restricted, this theorem holds for arbitrary Coxeter groups.
(In fact, a technical assumption that the canonical generating set of $W$ can be well-ordered is put in order to avoid a use of Axiom of Choice, but this assumption can be frequently satisfied in practical situations; e.g., in the case of countable ranks.)
The generality of our results would enable us to extend the group-theoretic studies of Coxeter groups to infinite rank cases, which will be a future research topic.

Finally, we give a brief remark on the importance of Coxeter groups of infinite ranks.
It is known that, even if a Coxeter group $W$ has finite rank, a subgroup of $W$ which also admits Coxeter group structure may have an infinite rank.
An example is found in a semidirect factor of the normalizer \cite{BH99} or of the centralizer \cite{Nui11} of a parabolic subgroup in $W$.
Hence Coxeter groups of infinite ranks may appear in a natural context of a study of Coxeter groups of finite ranks, therefore the properties of Coxeter groups of infinite ranks are still worthy to investigate even when our interest is restricted to the case of finite ranks.

This paper is organized as follows.
Section \ref{sec:Coxeter_group} summarizes basic definitions and facts about Coxeter groups.
We also present some auxiliary properties there.
In Section \ref{sec:intersection_parabolic}, we revisit the parabolic closures in Coxeter groups, and give the above-mentioned counterexample to show that the parabolic closure of a subset does not necessarily exist as a parabolic subgroup.
In Section \ref{sec:parabolic_closure}, we introduce the notions of locally parabolic subgroups and locally parabolic closures, and prove that locally parabolic closures always exist as locally parabolic subgroups.
We also give an example that locally parabolic closures and parabolic closures are different in general.
Finally, in Section \ref{sec:locally_finite}, we show that each locally finite subgroup of a Coxeter group is always contained in a maximal locally finite subgroup, which is locally parabolic.

\section{Coxeter groups}
\label{sec:Coxeter_group}

In this section, we summarize some basic definitions and facts about Coxeter groups; we refer to the book \cite{Hum_book} for those not stated explicitly in this paper.
We also present some further properties.
A pair $(W,S)$ of a group $W$ and its generating set $S$ is called a {\em Coxeter system} if $W$ admits the following presentation:
\begin{equation}
W = \langle S \mid (st)^{m(s,t)} = 1 \mbox{ for } s,t \in S \mbox{ such that } m(s,t) < \infty \rangle \enspace,
\end{equation}
where the indices $m(s,t)$ satisfy that $m(s,t) = m(t,s) \in \{1,2,\dots\} \cup \{\infty\}$, and $m(s,t) = 1$ if and only if $s = t$.
In this case, $W$ is called a {\em Coxeter group}.
Throughout this paper, $(W,S)$ denotes a Coxeter system unless otherwise specified.
Although several notions regarding Coxeter groups $W$ (e.g., reflections and parabolic subgroups) in fact depend on the choice of the generating set $S$, we often leave the $S$ implicit provided it is clear from the context.
We emphasize that the {\em rank} $|S|$ of $W$ is not restricted to be finite in this paper.
Let $V$ denote the {\em geometric representation space} of $W$ with bilinear form $\langle \cdot,\cdot \rangle$, and let $\Pi = \{\alpha_s \mid s \in S\}$ be its canonical $\mathbb{R}$-basis.
For an element $v$ of $V$, let $\mathrm{supp}(v)$ and $\mathrm{supp}_+(v)$ be the sets of all $s \in S$ such that the coefficient of $\alpha_s \in \Pi$ in $v$ is nonzero and is positive, respectively.
Let $\Phi \subseteq V$ denote the corresponding {\em root system}, which decomposes as $\Phi = \Phi^+ \cup \Phi^-$ into the sets of {\em positive roots} and of {\em negative roots}, respectively.
The three conditions $\gamma \in \Phi^+$, $\mathrm{supp}_+(\gamma) \neq \emptyset$ and $\mathrm{supp}(\gamma) = \mathrm{supp}_+(\gamma)$ are equivalent for any root $\gamma \in \Phi$.
For each $\gamma \in \Phi$, let $s_{\gamma}$ denote the corresponding {\em reflection} in $W$.
For each $I \subseteq S$, the subgroup $W_I = \langle I \rangle$ of $W$ generated by $I$ is called a {\em standard parabolic subgroup}, and any conjugate of $W_I$ in $W$ is called a {\em parabolic subgroup}.
We write $\Pi_I = \{\alpha_s \in \Pi \mid s \in I\}$, $\Phi_I = W_I \cdot \Pi_I$, and $\Phi_I^{\pm} = \Phi_I \cap \Phi^{\pm}$, respectively.
Here we emphasize the following easy fact, which will be used later:
\begin{lemma}
\label{lem:finite_subset_in_finite_rank}
Any finite subset of $W$ is contained in a standard parabolic subgroup $W_I$ with $I$ finite.
Similarly, any finitely generated subgroup of $W$ is contained in a standard parabolic subgroup $W_I$ with $I$ finite.
\end{lemma}
\begin{proof}
For the first part, define the $I \subseteq S$ to be the set of all standard generators $s \in S$ appearing in a fixed expression of some element of the given finite subset of $W$.
For the second part, apply the first part to a finite generating set of the given subgroup.
\end{proof}

Let $G$ be a {\em reflection subgroup} of $W$, i.e., subgroup generated by reflections.
We define $\Phi(G)$ to be the set of all roots $\gamma \in \Phi$ such that $s_{\gamma} \in G$.
Note that $G$ is generated by $\{s_{\gamma} \mid \gamma \in \Phi(G)\}$, $- \Phi(G) = \Phi(G)$, and $w \cdot \Phi(G) = \Phi(G)$ for every $w \in G$.
Moreover, we define $\Pi(G)$ to be the set of all \lq\lq simple roots'' in the \lq\lq root system'' $\Phi(G)$, namely the set of all roots $\gamma \in \Phi(G) \cap \Phi^+$ such that the conditions $\gamma = \sum_{i=1}^{k} c_i \beta_i$, $c_i > 0$ and $\beta_i \in \Phi(G) \cap \Phi^+$ imply $\beta_i = \gamma$ for every $i$.
Put $S(G) = \{s_{\gamma} \mid \gamma \in \Pi(G)\}$.
It was shown by Deodhar \cite{Deo89} and independently by Dyer \cite{Dye90} that $(G,S(G))$ is a Coxeter system.
We regard $S(G)$ as the canonical generating set of the Coxeter group $G$, and we call $|S(G)|$ the {\em rank} of $G$.
The following property is well-known:
\begin{lemma}
\label{lem:root_system_reflection_subgroup}
Let $G$ be a reflection subgroup of $W$.
Then $\Phi(G) = G \cdot \Pi(G)$.
\end{lemma}
Here we include a proof of Lemma \ref{lem:root_system_reflection_subgroup} for completeness, which uses the following result proven by Deodhar \cite[Theorem 5.4]{Deo82}:
\begin{lemma}
[Deodhar]
\label{lem:Deodhar_involution}
Let $w \in W$ such that $w^2 = 1$.
Then $w$ is a product of mutually commuting reflections in $W$.
\end{lemma}
\begin{proof}
[Proof of Lemma \ref{lem:root_system_reflection_subgroup}]
The non-trivial part is to show that $\gamma \in G \cdot \Pi(G)$ for each $\gamma \in \Phi(G)$.
As $s_{\gamma} \in G$ is an involution in the Coxeter group $G$, Lemma \ref{lem:Deodhar_involution} implies that there are $g_i \in G$ and $s_{\beta_i} \in S(G)$ ($1 \leq i \leq k$) such that $s_{\gamma} = g_1 s_{\beta_1} g_1^{-1} \cdots g_k s_{\beta_k} g_k^{-1}$ and each pair $g_i s_{\beta_i} g_i^{-1}$, $g_j s_{\beta_j} g_j^{-1}$ ($i \neq j$) are distinct and commute with each other.
This implies that $s_{\gamma} = s_{\gamma_1} \cdots s_{\gamma_k}$ and $\gamma_1,\dots,\gamma_k$ are mutually orthogonal, where $\gamma_i = g_i \cdot \beta_i \in G \cdot \Pi(G)$.
Therefore we have $s_{\gamma} \cdot \gamma_i = - \gamma_i$ and $\gamma_i = \pm \gamma$ for every $i$.
This implies that $k = 1$ and $\gamma = \pm \gamma_1 \in G \cdot \Pi(G)$, as desired.
\end{proof}
The following fundamental fact about reflection subgroups will be used below (see e.g., \cite[Theorem 2.3(iv)]{Nui11}):
\begin{lemma}
\label{lem:simple_root_to_negative}
For an element $1 \neq w \in G$ of a reflection subgroup $G$ of $W$ and a $\gamma \in \Pi(G)$, we have $w \cdot \gamma \in \Phi^-$ if and only if the length of $w s_{\gamma} \in G$ with respect to the generating set $S(G)$ of $G$ is shorter than that of $w$.
\end{lemma}
Now we include for completeness a proof of the following well-known basic fact:
\begin{lemma}
\label{lem:roots_for_parabolic}
For any $I \subseteq S$, we have $\Pi(W_I) = \Pi_I$, $S(W_I) = I$, and $\Phi(W_I) = \{\gamma \in \Phi \mid \mathrm{supp}(\gamma) \subseteq I\} = \Phi_I$.
\end{lemma}
\begin{proof}
First we have $\Phi(W_I) \subseteq \{\gamma \in \Phi \mid \mathrm{supp}(\gamma) \subseteq I\}$, as no element of $W_I = \langle I \rangle$ maps a positive root $\gamma$ such that $\mathrm{supp}(\gamma) \not\subseteq I$ to a negative root.
It then follows immediately that $\Pi(W_I) = \Pi_I$ and $S(W_I) = I$, therefore we have $\Phi(W_I) = \Phi_I$ by Lemma \ref{lem:root_system_reflection_subgroup}.
Finally, assume for contrary that $\gamma \in \Phi$ and $\mathrm{supp}(\gamma) \subseteq I$ but $s_{\gamma} \not\in W_I$.
Let $s_{\gamma} = s_1s_2 \cdots s_n$ be a shortest expression of $s_{\gamma} \in W$ with $s_1,\dots,s_n \in S$.
Take the maximal index $k \leq n$ such that $s_k \not\in I$ (which exists by the assumption on $\gamma$) and put $u = s_{k+1} \cdots s_n \in W_I$.
Then we have $s_{\gamma} u^{-1} \cdot \alpha_{s_k} \in \Phi^-$ by Lemma \ref{lem:simple_root_to_negative}.
However, this is impossible, as the actions by $s_{\gamma}$ and by $u^{-1}$ do not affect the coefficient of $\alpha_{s_k} \not\in \Pi_I$.
Hence we have $\Phi(W_I) \supseteq \{\gamma \in \Phi \mid \mathrm{supp}(\gamma) \subseteq I\}$, concluding the proof of Lemma \ref{lem:roots_for_parabolic}.
\end{proof}

Note that each parabolic subgroup $w W_I w^{-1}$ of $W$ is also a reflection subgroup, hence it has a generating set $S(w W_I w^{-1})$.
Now we show the following relation:
\begin{lemma}
\label{lem:canonical_generator_parabolic}
Let $G = w W_I w^{-1}$ be a parabolic subgroup of $W$ with $w \in W$ and $I \subseteq S$.
Then we have $\Pi(G) = w^I \cdot \Pi_I$ and $S(G) = w^I I (w^I)^{-1}$, where $w^I$ denotes the unique shortest element in the coset $w W_I$.
\end{lemma}
\begin{proof}
As shown in \cite[Section 5.12]{Hum_book}, such an element $w^I$ is well-defined and satisfies that $w_I = (w^I)^{-1} w \in W_I$ and $w^I \cdot \Pi_I \subseteq \Phi^+$.
Note that $G = w^I W_I (w^I)^{-1}$ and $w^I \cdot \Phi_I^{\pm} \subseteq \Phi^{\pm}$, respectively.
Put $u = w^I$.
Now let $\gamma \in \Phi(G) \cap \Phi^+$.
Then $s_{u^{-1} \cdot \gamma} = u^{-1} s_{\gamma} u \in u^{-1} G u = W_I$ and $u^{-1} \cdot \gamma \in \Phi(W_I) = \Phi_I$ (see Lemma \ref{lem:root_system_reflection_subgroup}).
We have $u^{-1} \cdot \gamma \in \Phi_I^+$, as otherwise $\gamma = u \cdot (u^{-1} \cdot \gamma) \in \Phi^-$, contradicting the choice of $\gamma$.
Hence $u^{-1} \cdot (\Phi(G) \cap \Phi^+) \subseteq \Phi_I^+$.
Similarly, we have $u \cdot \Phi_I^+ \subseteq \Phi(G) \cap \Phi^+$, therefore $u \cdot \Phi_I^+ = \Phi(G) \cap \Phi^+$.
This implies that a root $\gamma \in \Phi(G) \cap \Phi^+$ admits a non-trivial positive linear decomposition into elements of $\Phi(G) \cap \Phi^+$ if and only if $u^{-1} \cdot \gamma \in \Phi_I^+$ admits such a decomposition into elements of $\Phi_I^+$.
Hence we have $u \cdot \Pi_I = \Pi(G)$, therefore $u I u^{-1} = S(G)$, concluding the proof of Lemma \ref{lem:canonical_generator_parabolic}.
\end{proof}
This lemma implies that, for any expression $G = w W_I w^{-1}$ of a parabolic subgroup $G$ with $w \in W$ and $I \subseteq S$, the rank of $G$ as a reflection subgroup coincides with $|I|$, hence $|I|$ does not depend on the choice of the expression of $G$.

We also present the following fact, which will be used in later sections:
\begin{lemma}
\label{lem:generator_inclusion}
Let $H \leq G$ be two reflection subgroups of $W$.
Then $\Pi(G) \cap \Phi(H) \subseteq \Pi(H)$ and $S(G) \cap H \subseteq S(H)$.
\end{lemma}
\begin{proof}
The latter claim $S(G) \cap H \subseteq S(H)$ follows immediately from the former claim $\Pi(G) \cap \Phi(H) \subseteq \Pi(H)$.
To prove that $\Pi(G) \cap \Phi(H) \subseteq \Pi(H)$, let $\gamma \in \Pi(G) \cap \Phi(H)$.
If $\gamma = \sum_{i=1}^{k} \lambda_i \beta_i$ for $\lambda_i > 0$ and $\beta_i \in \Phi(H) \cap \Phi^+$, then we have $\beta_i \in \Phi(G) \cap \Phi^+$ as $H \leq G$, while $\gamma \in \Pi(G)$, therefore $\beta_i = \gamma$ for every $i$.
This means that $\gamma \in \Pi(H)$, therefore the claim holds.
\end{proof}

The {\em Coxeter graph} $\Gamma$ of a Coxeter system $(W,S)$ is a simple undirected graph with vertex set $S$ such that two vertices $s,t$ are joined by an edge with label $m(s,t)$ if $3 \leq m(s,t) \leq \infty$ and $s,t$ are not joined if $m(s,t) = 2$.
The edge labels $m(s,t)$ are often omitted when $m(s,t) = 3$.
Moreover, the {\em odd Coxeter graph} $\Gamma^{\mathrm{odd}}$ of $(W,S)$ is defined to be the subgraph of $\Gamma$ obtained by removing all edges with label $m(s,t)$ being either an even number or $\infty$.
The next lemma will be used later:
\begin{lemma}
\label{lem:root_and_odd_graph}
For any $s \in S$, let $I \subseteq S$ be the connected component of $\Gamma^{\mathrm{odd}}$ that contains $s$.
Then we have $\mathrm{supp}(w \cdot \alpha_s) \cap I \neq \emptyset$ for every $w \in W$.
\end{lemma}
\begin{proof}
First note the well-known fact that two generators $s,t \in S$ are conjugate in $W$ if and only if these belong to the same connected component of $\Gamma^{\mathrm{odd}}$.
Now assume for contrary that $\mathrm{supp}(w \cdot \alpha_s) \cap I = \emptyset$.
Then we have $w \cdot \alpha_s \in \Phi_{S \setminus I}$ by Lemma \ref{lem:roots_for_parabolic}, therefore we have $uw \cdot \alpha_s = \alpha_t$ for some $u \in W_{S \setminus I}$ and $t \in S \setminus I$.
This implies that $s$ and $t$ are conjugate in $W$, which contradicts the first remark and the definition of $I$.
Hence the claim holds.
\end{proof}

At last of this section, we introduce four special Coxeter systems (named $A_{\infty}^{\langle 1 \rangle}$, $A_{\infty}^{\langle 2 \rangle}$, $B_{\infty}$ and $D_{\infty}$ in this paper) which are irreducible and of infinite ranks, by the Coxeter graphs in Figure \ref{fig:locally_finite_Coxeter_group}.
The characteristics of these Coxeter systems are specified by the following proposition, which was proven by the author in \cite[Proposition 4.14]{Nui06} for the case that the group $G$ in the statement is irreducible.
Recall that a group $G$ is called {\em locally finite} if every finite subset of $G$ generates a finite subgroup.
Then the proposition is the following:
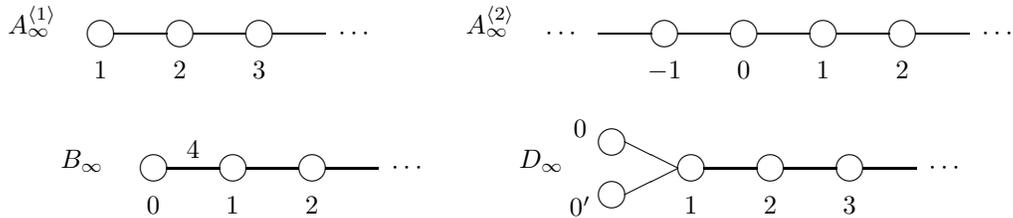
\begin{figure}[htb]
\centering
\begin{picture}(150,50)(0,-50)
\put(0,-20){$A_{\infty}^{\langle 1 \rangle}$}
\multiput(40,-20)(30,0){3}{\hbox to0pt{\hss\circle{10}\hss}\line(1,0){20}}
\put(125,-22){$\cdots$}
\put(35,-37){%
\hbox to0pt{\hss$1$\hss}\hspace*{30pt}%
\hbox to0pt{\hss$2$\hss}\hspace*{30pt}%
\hbox to0pt{\hss$3$\hss}%
}
\end{picture}
\qquad
\begin{picture}(220,50)(0,-50)
\put(0,-20){$A_{\infty}^{\langle 2 \rangle}$}
\put(30,-22){$\cdots$}
\put(70,-20){\line(-1,0){20}}
\multiput(80,-20)(30,0){4}{\hbox to0pt{\hss\circle{10}\hss}\line(1,0){20}}
\put(195,-22){$\cdots$}
\put(75,-37){%
\hbox to0pt{\hss$-1$\hss}\hspace*{30pt}%
\hbox to0pt{\hss$0$\hss}\hspace*{30pt}%
\hbox to0pt{\hss$1$\hss}\hspace*{30pt}%
\hbox to0pt{\hss$2$\hss}%
}
\end{picture}
\\
\begin{picture}(150,50)(0,-50)
\put(0,-20){$B_{\infty}$}
\multiput(40,-20)(30,0){3}{\hbox to0pt{\hss\circle{10}\hss}\line(1,0){20}}
\put(125,-22){$\cdots$}
\put(50,-16){\hbox to0pt{\hss$4$\hss}}
\put(35,-37){%
\hbox to0pt{\hss$0$\hss}\hspace*{30pt}%
\hbox to0pt{\hss$1$\hss}\hspace*{30pt}%
\hbox to0pt{\hss$2$\hss}%
}
\end{picture}
\qquad
\begin{picture}(180,50)(0,-50)
\put(0,-20){$D_{\infty}$}
\put(40,-10){\hbox to0pt{\hss\circle{10}\hss}}
\put(40,-30){\hbox to0pt{\hss\circle{10}\hss}}
\put(40,-10){\line(2,-1){20}}
\put(40,-30){\line(2,1){20}}
\multiput(70,-20)(30,0){3}{\hbox to0pt{\hss\circle{10}\hss}\line(1,0){20}}
\put(155,-22){$\cdots$}
\put(23,-8){\hbox to0pt{\hss$0$\hss}}
\put(23,-38){\hbox to0pt{\hss$0'$\hss}}
\put(65,-37){%
\hbox to0pt{\hss$1$\hss}\hspace*{30pt}%
\hbox to0pt{\hss$2$\hss}\hspace*{30pt}%
\hbox to0pt{\hss$3$\hss}%
}
\end{picture}
\caption{Locally finite irreducible Coxeter groups of infinite ranks}
\label{fig:locally_finite_Coxeter_group}
\end{figure}
\begin{proposition}
\label{prop:locally_finite_Coxeter}
Let $G$ be a reflection subgroup of a Coxeter group $W$, therefore $(G,S(G))$ is a Coxeter system.
Then $G$ is locally finite if and only if each irreducible component of $(G,S(G))$ is either finite or of one of the types $A_{\infty}^{\langle 1 \rangle}$, $A_{\infty}^{\langle 2 \rangle}$, $B_{\infty}$ and $D_{\infty}$.
\end{proposition}
\begin{proof}
We may assume without loss of generality that $G = W$.
First, Lemma \ref{lem:finite_subset_in_finite_rank} implies that $W$ is locally finite if and only if every standard parabolic subgroup $W_I$ of $W$ of finite rank is finite.
Moreover, $W$ satisfies the latter condition if and only if every irreducible component $W'$ of $W$ satisfies the same condition.
Finally, as mentioned above, the preceding result \cite[Proposition 4.14]{Nui06} implies that $W'$ satisfies this condition if and only if either $W'$ is finite or $W'$ is of one of the four types specified in the statement.
Hence the claim holds.
\end{proof}
The four Coxeter systems in Figure \ref{fig:locally_finite_Coxeter_group} will be studied again in later sections.

\section{On infinite intersections of parabolic subgroups}
\label{sec:intersection_parabolic}

Conventionally, the {\em parabolic closure} $P(X)$ of a subset $X$ of a Coxeter group $W$ is defined to be the intersection of the family of all parabolic subgroups of $W$ containing $X$ (note that the family is non-empty, as $W$ itself is a parabolic subgroup).
It is generally true that $P(X)$ is a subgroup of $W$ containing $X$, and under a certain condition (for example, when $W$ has finite rank) it holds further that $P(X)$ becomes the unique minimal {\em parabolic} subgroup containing $X$.
The latter property (of which we include a proof below for completeness) follows from the following fact, which is just a restatement of the result \cite[Corollary 7]{FH01} given by Franzsen and Howlett:
\begin{lemma}
[Franzsen--Howlett]
\label{lem:Fra-How}
Let $I,J \subseteq S$ and $w \in W$.
Then we have $W_I \cap w W_J w^{-1} = u W_K u^{-1}$ for some $K \subseteq I$ and $u \in W_I$.
Moreover, if in addition $W_I \not\subseteq w W_J w^{-1}$, then we have $K \subsetneq I$.
\end{lemma}
\begin{corollary}
\label{cor:intersection_parabolic_finite_rank}
\begin{enumerate}
\item \label{item:cor_intersection_parabolic_finite_rank_1}
If $G_1 \supsetneq G_2 \supsetneq \cdots$ is a strictly descending sequence of parabolic subgroups of $W$ such that the rank of $G_1$ is finite, then the length of this sequence is not larger than the rank of $G_1$ (hence finite).
\item \label{item:cor_intersection_parabolic_finite_rank_2}
If a (possibly infinite) family $\mathcal{F}$ consists of parabolic subgroups of $W$ and some member of $\mathcal{F}$ has finite rank, then $\bigcap \mathcal{F}$ is also a parabolic subgroup of $W$ of finite rank.
\item \label{item:cor_intersection_parabolic_finite_rank_3}
If $X$ is a subset of $W$ contained in some parabolic subgroup of finite rank, then $P(X)$ is the unique minimal parabolic subgroup of $W$ containing $X$ and the rank of $P(X)$ is finite.
\end{enumerate}
\end{corollary}
\begin{proof}
Lemma \ref{lem:Fra-How} says that the intersection of a parabolic subgroup of finite rank with another parabolic subgroup (possibly of infinite rank) not containing the former is again parabolic of smaller rank.
Hence the claim \ref{item:cor_intersection_parabolic_finite_rank_1} follows, as the ranks of the $G_i$ are strictly descending as well.
For the claim \ref{item:cor_intersection_parabolic_finite_rank_2}, beginning with any member $G_1$ of $\mathcal{F}$ of finite rank, if there is a member $H$ of $\mathcal{F}$ that does not contain $G_i$, then we define $G_{i+1}$ to be the intersection $G_i \cap H$.
By the claim \ref{item:cor_intersection_parabolic_finite_rank_1}, this inductive process of choosing descending parabolic subgroups halts in finitely many steps, and the last parabolic subgroup $G_i$ is nothing but the intersection $\bigcap \mathcal{F}$ (as it is contained in every member of $\mathcal{F}$) and it has smaller rank than $G_1$.
Hence the claim \ref{item:cor_intersection_parabolic_finite_rank_2} holds.
Finally, for the claim \ref{item:cor_intersection_parabolic_finite_rank_3}, the uniqueness property follows immediately from Lemma \ref{lem:Fra-How}, and the other part is implied by the claim \ref{item:cor_intersection_parabolic_finite_rank_2} for $\mathcal{F}$ being the family of all parabolic subgroups of $W$ containing $X$.
\end{proof}
Note that the assumption on the subset $X$ in the claim \ref{item:cor_intersection_parabolic_finite_rank_3} of Corollary \ref{cor:intersection_parabolic_finite_rank} is equivalent to that $X$ is contained in some {\em standard} parabolic subgroup of finite rank (apply Lemma \ref{lem:finite_subset_in_finite_rank} to the given parabolic subgroup of finite rank).
In fact most of the preceding works concerning parabolic closures of subgroups have dealt with situations where the assumption in the claim \ref{item:cor_intersection_parabolic_finite_rank_3} of Corollary \ref{cor:intersection_parabolic_finite_rank} is satisfied (e.g., $W$ is of finite rank, or $\langle X \rangle$ is finitely generated), and in those works the property that the parabolic closure is indeed a parabolic subgroup has played a significant role.
Hence it is natural to hope to have the desirable properties in Corollary \ref{cor:intersection_parabolic_finite_rank} for a general case as well.
However, unfortunately there is a counterexample for the claims when the finiteness assumptions in these claims are relaxed.
In the rest of this section we present such a counterexample, where (for the claim \ref{item:cor_intersection_parabolic_finite_rank_3}) the subgroup $\langle X \rangle$ is even a reflection subgroup and irreducible (as a Coxeter group).
\begin{example}
\label{ex:counterexample}
Let $(W,S)$ be the Coxeter system of type $A_{\infty}^{\langle 1 \rangle}$ with generating set $S = \{s_i \mid 1 \leq i \in \mathbb{Z}\}$ (see Figure \ref{fig:locally_finite_Coxeter_group} in Section \ref{sec:Coxeter_group}).
We write $\alpha_i = \alpha_{s_i}$ for simple roots of $(W,S)$.
Put $I_{i,j} = \{s_k \mid i \leq k \leq j\}$ for $i,j \in \mathbb{Z}$, and put $I_{i,\infty} = \{s_k \mid i \leq k\}$.
Put $\beta_i = \alpha_{2i-1} + \alpha_{2i} \in \Phi^+$ for $1 \leq i \in \mathbb{Z}$.
We define $X = \{s_{\beta_i} \mid i \geq 1\}$ and define a reflection subgroup $G$ of $W$ by $G = \langle X \rangle$.
On the other hand, let $\Psi_i = \{\beta_j \mid 1 \leq j \leq i\}$ and $X_i = \{s_{\beta_j} \mid 1 \leq j \leq i\}$ for $i \geq 1$.
The following property plays a key role in this construction:
\begin{equation}
\label{eq:counterexample_relation}
\begin{split}
&\mbox{If } u_i = (s_is_{i-1} \cdots s_1) \cdot \cdots \cdot (s_{2i-3}s_{2i-4}s_{2i-5}) \cdot (s_{2i-2}s_{2i-3}) \cdot (s_{2i-1}) \in W, \\
&\mbox{then } u_i \cdot \beta_j = \alpha_{j+i} \mbox{ and } u_i s_{\beta_j} u_i^{-1} = s_{j+i} \mbox{ for each } \beta_j \in \Psi_i.
\end{split}
\end{equation}
The property (\ref{eq:counterexample_relation}) implies that each $(\langle X_i \rangle,X_i)$ is a Coxeter system isomorphic to $(W_{I_{i+1,2i}},I_{i+1,2i})$.
More precisely, it follows from (\ref{eq:counterexample_relation}) that $u_i^{-1} \cdot \Pi_{I_{i+1,2i}} = \Psi_i \subseteq \Phi^+$, therefore $(u_i^{-1})^{I_{i+1,2i}} = u_i^{-1}$ and $\Pi(\langle X_i \rangle) = u_i^{-1} \cdot \Pi_{I_{i+1,2i}} = \Psi_i$ by Lemma \ref{lem:canonical_generator_parabolic} (hence $S(\langle X_i \rangle) = X_i$).
Now by taking the inductive limit, it follows that $(G,X)$ is a Coxeter system (of type $A_{\infty}^{\langle 1 \rangle}$, which is thus irreducible).
Moreover, as $G = \bigcup_{i \geq 1} \langle X_i \rangle$, Lemma \ref{lem:generator_inclusion} implies that
\begin{equation}
S(G) = S(G) \cap G = \bigcup_{i \geq 1} (S(G) \cap \langle X_i \rangle) \subseteq \bigcup_{i \geq 1} S(\langle X_i \rangle) = \bigcup_{i \geq 1} X_i = X \enspace,
\end{equation}
therefore $S(G) \subseteq X$, hence $S(G) = X$ (note that any proper subset of $X$ cannot generate $G$).
Now Lemma \ref{lem:canonical_generator_parabolic} implies that $G$ is not a parabolic subgroup, as $S(G)$ is not conjugate to any subset of $S$.

For each $i \geq 1$, we define $G_i$ to be the subgroup generated by $Y_i = X_i \cup I_{2i+1,\infty}$.
The same argument as the previous paragraph implies that $(G_i,Y_i)$ is also a Coxeter system and we have $\Pi(G_i) = u_i^{-1} \cdot \Pi_{I_{i+1,\infty}} = \Psi_i \cup \Pi_{I_{2i+1,\infty}}$.
Hence each $G_i$ is a parabolic subgroup, and we have $G_i \supsetneq G_{i+1}$ and $G_i \supseteq G$ for every $i$.

From now, we prove that $\bigcap_{i \geq 1} G_i = G$.
Let $w \in \bigcap_{i \geq 1} G_i$.
Take an index $i \in \mathbb{Z}$ such that $w \in W_{I_{1,2i}}$.
Now we show that $w \in \langle X_i \rangle$.
Assume the contrary, which means $w \in G_i \setminus \langle X_i \rangle$.
Now take a shortest expression $w = s_{\gamma_1} s_{\gamma_2} \cdots s_{\gamma_k}$ of $w$ as the product of some elements $s_{\gamma_j} \in Y_i$ with $\gamma_j \in \Phi^+$.
By multiplying an element of $\langle X_i \rangle$ if necessary, we may assume without loss of generality that $s_{\gamma_k} \in I_{2i+1,\infty}$.
Then it follows from \cite[Theorem 2.3(iv)]{Nui11} that $w \cdot \gamma_k \in \Phi^-$.
This is a contradiction, as we have $w \in W_{I_{1,2i}}$ and $\gamma_k \in \Pi_{I_{2i+1,\infty}}$.
Hence we have $w \in \langle X_i \rangle \subseteq G$, therefore the claim of this paragraph holds.

The above argument indeed gives a counterexample for the generalized version of Corollary \ref{cor:intersection_parabolic_finite_rank}.
Namely, for the claim \ref{item:cor_intersection_parabolic_finite_rank_1}, we have a strictly descending infinite sequence $G_1 \supsetneq G_2 \supsetneq \cdots$ of parabolic subgroups.
For the claims \ref{item:cor_intersection_parabolic_finite_rank_2} and \ref{item:cor_intersection_parabolic_finite_rank_3}, let $\mathcal{F}$ be the family of all parabolic subgroups of $W$ containing $X$ (hence $P(X) = \bigcap \mathcal{F}$).
Then $G_i \in \mathcal{F}$ for every $i \geq 1$, therefore $G \subseteq \bigcap \mathcal{F} \subseteq \bigcap_{i \geq 1} G_i = G$.
This implies that $G = \bigcap_{i \geq 1} \mathcal{F} = P(X)$, which is not a parabolic subgroup as shown above.
Hence the claims of Corollary \ref{cor:intersection_parabolic_finite_rank} will fail when the finiteness assumptions in the statements are removed.
\end{example}

\section{Locally parabolic subgroups and locally parabolic closures}
\label{sec:parabolic_closure}

As shown in Section \ref{sec:intersection_parabolic}, some desirable properties of parabolic closures $P(X)$ under some finiteness assumptions are lost in a general case, namely $P(X)$ itself is not always a parabolic subgroup.
In this section, we propose a modified definition of closure which has a similar desirable property, by introducing a notion of \lq\lq locally parabolic subgroups'' in Coxeter groups.
This closure is named \lq\lq locally parabolic closure'' and satisfies that it is indeed a locally parabolic subgroup (in contrast to the case of parabolic closure which is not necessarily a parabolic subgroup).

Our definition of locally parabolic subgroups is the following:
\begin{definition}
\label{defn:locally_parabolic}
We say that a subgroup $G$ of a Coxeter group $W$ is {\em locally parabolic} if $G$ is a reflection subgroup and each finite subset of the canonical generating set $S(G)$ of $G$ is conjugate in $W$ to a subset of $S$.
\end{definition}
By virtue of Lemma \ref{lem:canonical_generator_parabolic}, this definition is equivalent to that $G$ is a reflection subgroup and each finite subset of $S(G)$ generates a parabolic subgroup of $W$.
Lemma \ref{lem:canonical_generator_parabolic} also implies that any parabolic subgroup of $W$ is locally parabolic, hence Definition \ref{defn:locally_parabolic} generalizes the notion of parabolic subgroups.
Note also that a reflection subgroup of finite rank is locally parabolic if and only if it is parabolic, therefore the two notions agree on the finite rank case.
Now we present the following result, which corresponds to the claim \ref{item:cor_intersection_parabolic_finite_rank_2} of Corollary \ref{cor:intersection_parabolic_finite_rank} for parabolic subgroups but holds without any finiteness assumption:
\begin{theorem}
\label{thm:intersection_locally_parabolic}
The intersection of an arbitrary number of locally parabolic subgroups of a Coxeter group is also locally parabolic.
\end{theorem}
\begin{proof}
Let $P_{\lambda}$ ($\lambda \in \Lambda$) be locally parabolic subgroups of a Coxeter group $W$.
Put $P = \bigcap_{\lambda \in \Lambda} P_{\lambda} \leq W$.
First we show that $P$ is a reflection subgroup.
Let $w \in P$.
Then for each $\lambda \in \Lambda$, $w$ is contained in the subgroup $G_{\lambda}$ of $P_{\lambda}$ generated by a finite subset $X_{\lambda}$ of $S(P_{\lambda})$.
By the definition of locally parabolic subgroups, this $X_{\lambda}$ is conjugate in $W$ to a subset of $S$, therefore $G_{\lambda}$ is a parabolic subgroup of $W$.
This implies that the parabolic closure $P(w)$ of $w$ is contained in $G_{\lambda}$, therefore in $P_{\lambda}$.
Hence we have $w \in P(w) \leq P$, therefore $w$ is expressed as a product of reflections contained in $P(w) \subseteq P$.
This means that $P$ is a reflection subgroup, as desired.

From now, we show that $P$ is locally parabolic.
Let $X$ be a finite subset of $S(P)$.
By the same argument as the previous paragraph, the parabolic closure $P(X)$ of $X$ is contained in $P$.
Now we have $X \subseteq S(P) \cap P(X)$, while $S(P) \cap P(X) \subseteq S(P(X))$ by Lemma \ref{lem:generator_inclusion}, therefore $X \subseteq S(P(X))$.
Now $S(P(X))$ is conjugate to a subset of $S$ by Lemma \ref{lem:canonical_generator_parabolic}, so is $X$.
Hence $P$ is locally parabolic, concluding the proof of Theorem \ref{thm:intersection_locally_parabolic}.
\end{proof}
\begin{remark}
\label{rem:intersection_infinite_parabolic}
By Theorem \ref{thm:intersection_locally_parabolic}, the intersection of an arbitrary number of parabolic subgroups of a Coxeter group is locally parabolic, hence it is a reflection subgroup.
In contrast, the intersection of (even two) reflection subgroups is not necessarily a reflection subgroup.
For instance, consider the Weyl group $W$ of type $G_2$ (or finite Coxeter group of type $I_2(6)$) with canonical generators $s,t$ and their relation $(st)^6 = 1$.
Then the intersection of the reflection subgroups generated by $\{s,tstst\}$ and by $\{t,ststs\}$, respectively, is the subgroup generated by the longest element $(st)^3$ of $W$, which is not a reflection subgroup, as $(st)^3$ is central in $W$.
\end{remark}
As each subset $X$ of $W$ is contained in a locally parabolic subgroup of $W$ (e.g., $W$ itself), Theorem \ref{thm:intersection_locally_parabolic} implies that there exists a unique minimal locally parabolic subgroup of $W$ containing $X$.
We write it as $LP(X)$ and call it the {\em locally parabolic closure} of $X$.
We also write $LP(X)$ as $LP(x)$ when $X = \{x\}$.
The parabolic and locally parabolic closures satisfy the relation $LP(X) \subseteq P(X)$, as each parabolic subgroup is also a locally parabolic subgroup.
Moreover, this notion indeed generalizes the notion of parabolic closure, as shown by the following result:
\begin{lemma}
\label{lem:two_closures}
If $P \leq W$ is a parabolic subgroup of finite rank and $G$ is a locally parabolic subgroup of $W$ contained in $P$, then $G$ is also parabolic of rank not larger than that of $P$.
Hence we have $LP(X) = P(X)$ if $X$ is a subset of $W$ contained in some parabolic subgroup of finite rank.
\end{lemma}
\begin{proof}
The second part follows from the first part and the fact (by Corollary \ref{cor:intersection_parabolic_finite_rank}) that $P(X)$ is now a parabolic subgroup of finite rank.
For the first part, note that $|S(P)| < \infty$ by the assumption.
We show that $S(G)$ is also finite, which implies that $G$ is parabolic (by the definition of locally parabolic subgroups) and $|S(G)| \leq |S(P)|$ (by Lemma \ref{lem:Fra-How}), concluding the proof.
Assume for contrary that $|S(G)| = \infty$.
Then there is a finite subset $X$ of $S(G)$ such that $|X| > |S(P)|$.
As $G$ is locally parabolic, $\langle X \rangle$ is parabolic of rank $|X|$ greater than the rank $|S(P)|$ of $P$, while $\langle X \rangle \leq P$.
This contradicts Lemma \ref{lem:Fra-How}.
Hence the claim holds.
\end{proof}

In the rest of this section, we give an example to show that the relation $LP(X) = P(X)$ in the second part of Lemma \ref{lem:two_closures} does not necessarily hold when the assumption on $X$ is relaxed.
More precisely, we construct a locally parabolic subgroup $G$ of a Coxeter group $W$ (hence $LP(G) = G$) such that $G \neq W$ and $P(G) = W$, therefore $LP(G) \neq P(G)$.
In this case, the parabolic closure $P(G)$ does not provide much information on $G$, as $P(G) = W$.
This example and the above fact that locally parabolic closures are always locally parabolic would suggest an interpretation that locally parabolic closure is in fact more natural notion of closure than parabolic closure.
\begin{example}
\label{ex:two_closures_differ}
Let $(W,S)$ be an irreducible Coxeter system such that $S = \{s_i \mid 1 \leq i \in \mathbb{Z}\}$, $m(s_i,s_j) = 2$ if $|i-j| \geq 2$, and $m(s_i,s_j)$ is either $\infty$ or an even number larger than $2$ if $|i-j| = 1$.
Write $\alpha_i = \alpha_{s_i}$ ($i \geq 1$) for simplicity.
Put
\begin{equation}
\gamma_i = u_i \cdot \alpha_i \in \Phi \mbox{ for each } i \geq 1 \,,\, \mbox{where } u_i = s_1 s_2 \cdots s_{i+1} \enspace,
\end{equation}
and put $\Psi = \{\gamma_i \mid i \geq 1\}$ and $X = \{s_{\gamma_i} \mid i \geq 1\}$.
Moreover, for each $i \geq 1$, put $\Psi_i = \{\gamma_j \mid 1 \leq j \leq i\}$ and $X_i = \{s_{\gamma_j} \mid 1 \leq j \leq i\}$.
We show that the reflection subgroup $G$ generated by $X$ is the desired example.

First, note that each $\gamma_i$ is a positive root, as $\alpha_{i+1} \in \mathrm{supp}_+(\gamma_i)$ by the definition.
Secondly, we have
\begin{equation}
\label{eq:example_closure_relation}
\gamma_i = u_j \cdot \alpha_i \mbox{ for any } j \geq i \geq 1 \enspace,
\end{equation}
therefore $u_i^{-1} \cdot \Psi_i = \{\alpha_j \mid 1 \leq j \leq i\} \subseteq \Pi$ for each $i \geq 1$.
Now the same argument as the one following property (\ref{eq:counterexample_relation}) in Example \ref{ex:counterexample} implies that $S(G) \subseteq X$.
Moreover, we have $S(G) = X$; indeed, if $s_{\gamma_i} \in X \setminus S(G)$, then Lemma \ref{lem:simple_root_to_negative} implies that $s_{\gamma_i} \cdot \gamma_j \in \Phi^-$ for some $s_{\gamma_j} \in S(G)$.
However, this is impossible, as we have $\langle \gamma_i,\gamma_j \rangle = \langle \alpha_i,\alpha_j \rangle \leq 0$ by (\ref{eq:example_closure_relation}).
Hence $G$ is a Coxeter group with $S(G) = X$.
(Note that the relation $\langle \gamma_i,\gamma_j \rangle = \langle \alpha_i,\alpha_j \rangle$ implies that $(G,X)$ is isomorphic to $(W,S)$.)
Moreover, the property (\ref{eq:example_closure_relation}) implies that each $X_i \subseteq X$ is conjugate to $\{s_j \mid 1 \leq j \leq i\} \subseteq S$, therefore $G$ is locally parabolic.
This $G$ does not contain $s_1$, therefore $G \neq W$.
Indeed, if $s_1 \in G$, then we have $\alpha_1 \in \Pi \cap \Phi(G) \subseteq \Pi(G) = \Psi$ by Lemma \ref{lem:generator_inclusion}, while $\alpha_{i+1} \in \mathrm{supp}(\gamma_i)$ for each $\gamma_i \in \Psi$, a contradiction.

The remaining task is to show that $P(G) = W$.
Let $w W_I w^{-1}$ be any parabolic subgroup of $W$ containing $G$.
Then for each $\gamma_i \in \Psi$, we have $w^{-1} s_{\gamma_i} w \in w^{-1} G w \subseteq W_I$, therefore $s_{w^{-1} \cdot \gamma_i} \in W_I$ and $\mathrm{supp}(w^{-1} \cdot \gamma_i) \subseteq I$ by Lemma \ref{lem:roots_for_parabolic}.
Now by Lemma \ref{lem:root_and_odd_graph} and the definition of $\gamma_i$, the set $\mathrm{supp}(w^{-1} \cdot \gamma_i)$ intersects with the connected component of $\Gamma^{\mathrm{odd}}$ containing $s_i$.
By the definition of $(W,S)$, this connected component consists of $s_i$ only, therefore we have $s_i \in \mathrm{supp}(w^{-1} \cdot \gamma_i) \subseteq I$.
Hence we have $I = S$ and $w W_I w^{-1} = W$.
This implies that $P(G) = W$, as desired.
\end{example}

\section{On locally finite subgroups of Coxeter groups}
\label{sec:locally_finite}

In this section, we give generalizations of two well-known facts on finite subgroups and parabolic subgroups in Coxeter groups, to the case of locally parabolic subgroups instead of parabolic ones.
One of the original facts is the following, which was originally proven by Tits for finite rank case (see e.g., \cite[Theorem 4.5.3]{BB_book} for a proof) and then easily extended to arbitrary rank case (see e.g., Lemma \ref{lem:finite_subset_in_finite_rank}):
\begin{lemma}
[Tits]
\label{lem:Tits}
Each finite subgroup of $W$ is contained in a finite parabolic subgroup of $W$.
\end{lemma}
Another original fact is the following, which is well-known for experts especially in the finite rank case but seems rarely specified in preceding works explicitly:
\begin{proposition}
\label{prop:Coxeter_maximal_finite_subgroup}
Suppose that for each subset $I$ of $S$ such that $W_I$ is finite, there exists a subset $J$ of $S$ containing $I$ which is maximal subject to the condition $|W_J| < \infty$ (for example, this assumption is satisfied when $W$ has finite rank).
Then each finite subgroup of $W$ is contained in a maximal finite subgroup of $W$, which is parabolic.
\end{proposition}
Here we include a proof of this proposition for completeness.
The key property is the following result given by Franzsen and Howlett \cite[Lemma 8]{FH01}:
\begin{lemma}
[Franzsen--Howlett]
\label{lem:Fra-How_max_finite_standard}
If a subset $I \subseteq S$ satisfies that $W_I$ is finite and $I$ is maximal subject to the condition $|W_I| < \infty$, then $W_I$ is a maximal finite subgroup of $W$.
\end{lemma}
\begin{proof}
[Proof of Proposition \ref{prop:Coxeter_maximal_finite_subgroup}]
Let $G$ be a finite subgroup of $W$.
By Lemma \ref{lem:Tits}, there is a finite parabolic subgroup $P$ of $W$ such that $G \leq P$.
Write $P = w W_I w^{-1}$ with $I \subseteq S$ and $w \in W$.
Then $W_I$ is finite as well as $P$, therefore by the assumption, there is a subset $J \subseteq S$ containing $I$ such that $W_J$ is finite and $J$ is maximal subject to the condition $|W_J| < \infty$.
By Lemma \ref{lem:Fra-How_max_finite_standard}, $W_J$ is a maximal finite subgroup of $W$, so is $w W_J w^{-1}$.
Moreover, $w W_J w^{-1}$ is parabolic and contains $G$.
Hence the claim holds.
\end{proof}
From now, we give generalizations of Lemma \ref{lem:Tits} and Proposition \ref{prop:Coxeter_maximal_finite_subgroup}.
First, our generalization of Lemma \ref{lem:Tits} is the following:
\begin{theorem}
\label{thm:lp_closure_is_locally_finite}
Suppose that there exists a well-ordering $\preceq$ on $S$.
Let $G$ be a locally finite subgroup of $W$.
Then its locally parabolic closure $LP(G)$ is locally finite.
Hence $G$ is contained in a locally finite subgroup of $W$ which is locally parabolic.
\end{theorem}
\begin{remark}
\label{rem:AC}
A proof of Theorem \ref{thm:lp_closure_is_locally_finite} without the assumption on $S$ would require Axiom of Choice (AC).
The technical assumption that there exists a well-ordering $\preceq$ on $S$ is introduced for the purpose of clarifying in which part we need AC (the proof of Theorem \ref{thm:lp_closure_is_locally_finite} below itself does not use AC).
Note that this assumption may be frequently satisfied for a given concrete Coxeter system $(W,S)$.
A similar remark is also applied to Theorem \ref{thm:maximal_locally_finite_subgroup} below.
\end{remark}
In the proof of Theorem \ref{thm:lp_closure_is_locally_finite}, we use the following property:
\begin{lemma}
\label{lem:ascending_locally_parabolic}
Let $(\Lambda,\prec)$ denote a totally ordered set.
Let $G_{\lambda}$ be subgroups of $W$ indexed by elements $\lambda \in \Lambda$ in such a way that $G_{\lambda} \leq G_{\mu}$ whenever $\lambda \prec \mu$.
Put $G = \bigcup_{\lambda \in \Lambda} G_{\lambda}$.
\begin{enumerate}
\item \label{item:lem_ascending_lp_1}
$G$ is a subgroup of $W$.
\item \label{item:lem_ascending_lp_2}
If every $G_{\lambda}$ is a reflection subgroup, then so is $G$.
\item \label{item:lem_ascending_lp_3}
If every $G_{\lambda}$ is locally finite, then so is $G$.
\item \label{item:lem_ascending_lp_4}
If every $G_{\lambda}$ is locally parabolic, then so is $G$.
\end{enumerate}
\end{lemma}
\begin{proof}
(\ref{item:lem_ascending_lp_1})
This follows from the fact that any pair of $g_1,g_2 \in G$ is contained in some $G_{\lambda}$.

(\ref{item:lem_ascending_lp_2})
Each $g \in G$ belongs to some $G_{\lambda}$, therefore $g$ is a product of reflections contained in $G_{\lambda} \subseteq G$.

(\ref{item:lem_ascending_lp_3})
This follows from the fact that any finite subset of $G$ is contained in some $G_{\lambda}$.

(\ref{item:lem_ascending_lp_4})
First, $G$ is a reflection subgroup by the claim \ref{item:lem_ascending_lp_2}.
Let $X$ be a finite subset of $S(G)$.
Then $X \subseteq G_{\lambda}$ for some $\lambda \in \Lambda$.
Now we have $X \subseteq S(G) \cap G_{\lambda} \subseteq S(G_{\lambda})$ (see Lemma \ref{lem:generator_inclusion}), while $G_{\lambda}$ is locally parabolic, therefore $X$ is conjugate to a subset of $S$.
Hence $G$ is locally parabolic.
\end{proof}
\begin{proof}
[Proof of Theorem \ref{thm:lp_closure_is_locally_finite}]
First note that the well-ordering $\preceq$ on $S$ induces a well-ordering on the set $S^{<\omega}$ of all finite (possibly empty) sequences of elements of $S$, in such a way that $a \in S^{<\omega}$ is smaller than $b \in S^{<\omega}$ if and only if either the length of $a$ is shorter than $b$, or $a$ and $b$ have the same length and $a$ is smaller than $b$ with respect to the lexicographic order.
As there exists an obvious surjective map $S^{<\omega} \to W$, it follows (without Axiom of Choice) that there exists an injective map $W \to S^{<\omega}$, hence an injective map $G \to S^{<\omega}$.
This induces a well-ordering on $G$, therefore there exists the cardinal number $\kappa$ of $G$ and a bijection $\kappa \to G$, $\beta \mapsto g_{\beta}$ ($\beta < \kappa$).

By the previous paragraph, it holds that (*) there exists a cardinal number $\kappa$ and a (not necessarily bijective) map $\kappa \to G$, $\beta \mapsto g_{\beta}$ ($\beta < \kappa$) such that $G_{\kappa} = G$, where we put $G_{\alpha} = \langle \{g_{\beta} \mid \beta < \alpha\} \rangle \leq G$ for each ordinal number $\alpha \leq \kappa$.
We show that $LP(G)$ is locally finite, by transfinite induction on such a cardinal number $\kappa$.
If $\kappa$ is finite, then $G = G_{\kappa}$ is finitely generated as well as locally finite, therefore $G$ is finite.
Now we have $LP(G) = P(G)$ by Lemma \ref{lem:finite_subset_in_finite_rank} and Lemma \ref{lem:two_closures}, while $P(G)$ is finite by Lemma \ref{lem:Tits}.
Hence $LP(G)$ is locally finite, as desired.

From now, we consider the remaining case that $\kappa$ is (infinite and) a limit ordinal.
In this case, we have $G = G_{\kappa} = \bigcup_{\alpha < \kappa} G_{\alpha}$.
Indeed, any element $g \in G = G_{\kappa}$ is written as a product of finitely many elements $g_{\beta}$ with $\beta < \kappa$ and their inverses, therefore we have $g \in G_{\alpha + 1}$ and $\alpha + 1 < \kappa$ (as $\kappa$ is a limit ordinal) for $\alpha < \kappa$ being the maximum of the finitely many indices $\beta$.
Moreover, each subgroup $G_{\alpha}$ ($\alpha < \kappa$) of $G$ is locally finite as well as $G$, and this $G_{\alpha}$ also satisfies the condition (*) with the cardinal number $\kappa' \leq \alpha$ of $\alpha$ by taking composition of the above map $\beta \mapsto g_{\beta}$ and a bijection $\kappa' \to \alpha$.
Hence by the hypothesis of the transfinite induction, the locally parabolic closure $LP(G_{\alpha})$ of $G_{\alpha}$ is locally finite for every $\alpha < \kappa$.
As $(G_{\alpha})_{\alpha < \kappa}$ is a weakly ascending sequence, so is $(LP(G_{\alpha}))_{\alpha < \kappa}$.
By Lemma \ref{lem:ascending_locally_parabolic}, $L = \bigcup_{\alpha < \kappa} LP(G_{\alpha})$ is a locally finite and locally parabolic subgroup of $W$.
Moreover, we have $G = \bigcup_{\alpha < \kappa} G_{\alpha} \leq L$.
This implies that the locally parabolic closure $LP(G)$ of $G$ is a subgroup of $L$, which is locally finite as well as $L$.

Summarizing, $LP(G)$ is indeed locally finite in any case.
Hence the transfinite induction completes the proof of Theorem \ref{thm:lp_closure_is_locally_finite}.
\end{proof}

Secondly, we present a generalization of Proposition \ref{prop:Coxeter_maximal_finite_subgroup} as follows:
\begin{theorem}
\label{thm:maximal_locally_finite_subgroup}
Suppose that there exists a well-ordering $\preceq$ on $S$.
Then each locally finite subgroup of $W$ is contained in a maximal locally finite subgroup of $W$, which is locally parabolic.
\end{theorem}
\begin{proof}
Let $G$ be a locally finite subgroup of $W$.
First, if there exists a maximal locally finite subgroup $H$ of $W$ containing $G$, then by Theorem \ref{thm:lp_closure_is_locally_finite}, $LP(H)$ is also locally finite and $G \leq H \leq LP(H)$, therefore $LP(H) = H$ by the maximality of $H$ and $H$ is locally parabolic, as desired.
From now, we prove the existence of such a subgroup $H$.

By the same argument as the proof of Theorem \ref{thm:lp_closure_is_locally_finite}, the well-ordering on $S$ induces (without Axiom of Choice) a well-ordering on $W$.
Now assume for contrary that there does not exist a maximal locally finite subgroup of $W$ containing $G$.
Then we construct, by transfinite induction, locally finite subgroups $H_{\alpha} \leq W$ for ordinal numbers $\alpha$ such that $G \subseteq H_{\alpha} \subsetneq H_{\beta}$ for any $\alpha < \beta$.
We define $H_0 = G$.
Now suppose that $\alpha > 0$ is an ordinal number and such subgroups $H_{\beta}$ have been defined for all $\beta < \alpha$.
First we consider the case that $\alpha$ is a successor ordinal $\alpha_0 + 1$.
Then $H_{\alpha_0}$ is locally finite by the induction hypothesis.
On the other hand, by the assumption, $H_{\alpha_0}$ is not a maximal locally finite subgroup of $W$, therefore $H_{\alpha_0}$ is properly contained in some locally finite subgroup of $W$.
This implies that there is an element $g \in W \setminus H_{\alpha_0}$ such that $H_{\alpha_0} \cup \{g\}$ generates a locally finite subgroup.
We take the minimum (according to the above well-ordering on $W$) among such elements $g$, and define $H_{\alpha} = \langle H_{\alpha_0} \cup \{g\} \rangle$.
This $H_{\alpha}$ satisfies the desired conditions.
Secondly, we consider the other case that $\alpha$ is a limit ordinal.
We define $H_{\alpha}$ to be the union of all the subgroups $H_{\beta}$ with $\beta < \alpha$.
Then by the induction hypothesis and Lemma \ref{lem:ascending_locally_parabolic}, $H_{\alpha}$ is also a locally finite subgroup of $W$ containing $G$.
Moreover, we have $H_{\beta} \subsetneq H_{\beta + 1} \subseteq H_{\alpha}$ for every $\beta < \alpha$, as $\alpha$ is a limit ordinal.
Hence this $H_{\alpha}$ satisfies the desired conditions.
Therefore $H_{\alpha}$ has been successfully constructed in any case.

Now take a limit ordinal $\alpha$ such that there exists no injective map $\alpha \to W$ (for example, the supremum of the non-empty set of all ordinal numbers that are isomorphic to some well-ordering on $W$, which indeed exists under Zermelo-Fraenkel axioms, satisfies this condition).
Then we define a map $\varphi:\alpha \to W$ by $\varphi(\beta) = \min(H_{\beta + 1} \setminus H_{\beta})$ for $\beta < \alpha$.
Now the above properties of the subgroups $H_{\beta}$ imply that this $\varphi$ is well-defined and injective.
However, by the choice of $\alpha$, this is a contradiction.
Hence the proof of Theorem \ref{thm:maximal_locally_finite_subgroup} is concluded.
\end{proof}
Note that the structure of locally finite, locally parabolic subgroups appeared in the statements of Theorem \ref{thm:lp_closure_is_locally_finite} and Theorem \ref{thm:maximal_locally_finite_subgroup} is well restricted by Proposition \ref{prop:locally_finite_Coxeter} in Section \ref{sec:Coxeter_group}; namely, each irreducible component of such a subgroup is either finite or of one of the four types in Figure \ref{fig:locally_finite_Coxeter_group}.

\paragraph{Acknowledgements}
The author would like to thank the anonymous referee for his/her kind and careful review.


\begin{thebibliography}{99}

\bibitem{BB_book}
A.~Bj\"{o}rner, F.~Brenti,
\textit{Combinatorics of Coxeter groups},
Springer, New York, 2005.

\bibitem{BH99}
B.~Brink, R.~B.~Howlett,
Normalizers of parabolic subgroups in Coxeter groups,
\textit{Invent.\ Math.\ }\textbf{136} (1999), 323--352.

\bibitem{Deo82}
V.~V.~Deodhar,
On the root system of a Coxeter group,
\textit{Commun.\ Algebra} \textbf{10} (1982), 611--630.

\bibitem{Deo89}
V.~V.~Deodhar,
A note on subgroups generated by reflections in Coxeter groups,
\textit{Arch.\ Math.\ (Basel)} \textbf{53} (1989), 543--546.

\bibitem{Dye90}
M.~Dyer,
Reflection subgroups of Coxeter systems,
\textit{J.\ Algebra} \textbf{135} (1990), 57--73.

\bibitem{FH01}
W.~N.~Franzsen, R.~B.~Howlett,
Automorphisms of Coxeter groups of rank three,
\textit{Proc.\ Amer.\ Math.\ Soc.\ }\textbf{129} (2001), 2607--2616.

\bibitem{Hum_book}
J.~E.~Humphreys,
\textit{Reflection groups and Coxeter groups},
Cambridge University Press, Cambridge, 1990.

\bibitem{Kra94}
D.~Krammer,
The conjugacy problem for Coxeter groups,
Ph.D.\ thesis, University of Utrecht, Utrecht 1994.

\bibitem{Nui11}
K.~Nuida,
On centralizers of parabolic subgroups in Coxeter groups,
\textit{J.\ Group Theory} \textbf{14}(6) (2011), 891--930.

\bibitem{Nui06}
K.~Nuida,
On the direct indecomposability of infinite irreducible Coxeter groups and the Isomorphism Problem of Coxeter groups,
\textit{Commun.\ Algebra} \textbf{34}(7) (2006), 2559--2595.

\bibitem{Qi07}
D.~Qi,
A note on parabolic subgroups of a Coxeter group,
\textit{Expo.\ Math.\ }\textbf{25} (2007), 77--81.

\end{thebibliography}
\end{document}